\documentclass{amsart}
\usepackage{a4wide,amssymb}
\usepackage[all]{xy}
\usepackage{color}

\newcommand{\Tau}{\mathsf T\!}
\newcommand{\uhr}{\upharpoonright}
\newcommand{\w}{\omega}
\newcommand{\cl}{\operatorname{cl}}

\newcommand{\ka}{\kappa}

\newcommand{\R}{\mathbb{R}}

\newcommand{\E}{E_{a,b}(X)}
\newcommand{\D}{D_{a,b}(X)}

\DeclareSymbolFont{extraup}{U}{zavm}{m}{n}
\DeclareMathSymbol{\varheart}{\mathalpha}{extraup}{86}
\DeclareMathSymbol{\vardiamond}{\mathalpha}{extraup}{87}

\newtheorem{problem}{Problem}
\newtheorem{theorem}{Theorem}[section]
\newtheorem{proposition}[theorem]{Proposition}

\newtheorem{corollary}[theorem]{Corollary}

\newtheorem{lemma}[theorem]{Lemma}
\newtheorem*{claim}{Claim}

\theoremstyle{definition}
\newtheorem{remark}[theorem]{Remark}
\newtheorem{example}[theorem]{Example}

\title[On regular separable countably compact $\R$-rigid spaces]{On regular separable countably compact $\R$-rigid spaces}
\author[S. Bardyla]{Serhii Bardyla}
\author[L.~Zdomskyy]{Lyubomyr Zdomskyy}

\address{Universit\"at Wien, Institut f\"ur Mathematik, Kurt G\"odel Research Center, Kolingasse 14--16, 1090 Wien.}
\email{sbardyla@yahoo.com}
\urladdr{http://www.logic.univie.ac.at/$\tilde{\ }$bardylas55/}
\email{lzdomsky@gmail.com}
\urladdr{http://www.logic.univie.ac.at/$\tilde{\ }$lzdomsky/}

\thanks{The first author was supported by the Austrian Science Fund FWF (grants I 3709-N35 and M 2967).
The second author was supported by the Austrian Science Fund FWF (grants I 3709-N35 and I 2374-N35).}



\begin{document}
\begin{abstract}
A topological space $X$ is said to be {\em $Y$-rigid} if any continuous map $f:X\rightarrow Y$ is constant.
In this paper we construct a number of examples of regular countably compact $\mathbb R$-rigid spaces with additional properties like
separability and first countability. This way we answer several questions of Tzannes, Banakh, Ravsky, as well as
get a consistent example of $\mathbb R$-rigid Nyikos space. Also, we show that it is consistent with ZFC that for every cardinal $\kappa<\mathfrak c$ there exists a regular separable countably compact space $X$ which is $Y$-rigid with respect to any $T_1$ space $Y$ of pseudocharacter $\leq\kappa$.
\end{abstract}
\maketitle

\section{Introduction}


$Y$-rigid spaces are of particular interest in general topology, see, e.g., \cite{1, 2, CW, vD, 3.5, 4, 5, 6, 7, Tz,8}.
For a long time all known examples of regular $\R$-rigid spaces were far from being countably compact.
In~\cite{Tz} Tzannes constructed a Hausdorff countably compact $\R$-rigid space $T$ such that no pair of distinct points of $T$ have disjoint closed neighborhoods, i.e., this space is far from being regular. This motivated
Tzannes to ask the following two questions:

\begin{problem}[{\cite[Problem C65]{Prob}}]\label{p1}
Does there exist a regular (first countable, separable)
countably compact space on which every continuous real-valued function is constant?
\end{problem}

\begin{problem}[{\cite[Problem C66]{Prob}}]\label{p2}
Does there exist, for every Hausdorff space $R$, a regular
(first countable, separable) countably compact space on which every continuous
function into $R$ is constant?
\end{problem}

Recently, the first author jointly with Osipov obtained the following affirmative answer to both questions of Tzannes without the properties mentioned in brackets.

\begin{theorem}[{\cite[Theorem 3]{BO}}]\label{BO}
For each cardinal $\ka\geq\omega$ there exists a regular infinite  space $R_{\kappa}$
such that the closure in $R_\kappa$ of any $X\in [R_\kappa]^{\leq\kappa}$ is compact, and for
any $T_1$ space $Y$ of pseudocharacter $\psi(Y)\leq \kappa$ each continuous map $f:R_{\kappa}\rightarrow Y$ is constant.
\end{theorem}

However, for any $\kappa\geq\omega$ the space $R_{\kappa}$ is neither separable nor first countable. Since $\mathbb R$-rigid
spaces cannot be Tychonoff (and hence are non-compact), the ultimate version of Problem~\ref{p1} may be thought of as a strong form of the following
 notoriously difficult old problem of Nyikos, see ~\cite{Ny1,Ny2} for the history thereof and related results.

\begin{problem}[{\cite[Problem 1]{Ny2}}]\label{pN}
Does ZFC imply the existence of a regular separable first countable countably compact non-compact space?
\end{problem}

Further a regular separable first countable countably compact space will be called a {\em Nyikos} space.

In \cite{BBRk} Banakh, Ravsky and the first author posed the following problem, which is a weaker form of Nyikos'
Problem~\ref{pN}.

\begin{problem}[{\cite[Question 4.9]{BBRk}}]\label{p3}
Does there exist in ZFC an example of a separable regular sequentially compact space which
is not Tychonoff?
\end{problem}

In this paper we construct a ZFC example of a regular separable totally countably compact $\R$-rigid space, thus answering  Problem~\ref{p1} for separable spaces in the affirmative, see Theorem~\ref{t1}. On the other hand, it is easy to see that the answer to Problem~\ref{p2} for separable spaces is negative, see Remark~\ref{rem_sep}.  In Theorem~\ref{tt2} we show that it is consistent with ZFC that for every cardinal $\kappa<\mathfrak c$ there exists a regular separable totally countably compact space $X$ which is $Y$-rigid with respect to any $T_1$ space $Y$ of pseudocharacter $\leq\kappa$.
In addition we construct a consistent example of an $\R$-rigid Nyikos space, thus
providing a consistent counterexample to Problem~\ref{pN} with additional strong topological properties, see Theorem~\ref{t6}.
Finally, in Theorem~\ref{t4} we provide
 an affirmative answer to Problem~\ref{p3} outright in ZFC.

The notions used but not defined in this paper are standard and can be found in~\cite{int, Eng, Kun}.

\section{van Douwen's extension}\label{V}
For any distinct points $a,b$ of a topological space $X$ by $\operatorname{Const}(X)_{a,b}$ we denote the class of $T_1$ spaces $Y$ such that $f(a)=f(b)$ for any continuous mapping $f:X\rightarrow Y$.

In this section we slightly modify the construction of van Douwen~\cite{vD}. In particular, for any distinct non-isolated points $a,b$ of a separable regular space $X$ we shall construct a space $D_{a,b}(X)$ such that for any $Y\in \operatorname{Const}(X)_{a,b}$, each continuous function $h:D_{a,b}(X)\rightarrow Y$ is constant. Moreover, the space $D_{a,b}(X)$ remains regular and separable.

Let $X$ be a regular separable  space and $a,b\in X$ be non-isolated points.
For points $c,d\in X$ set $$\mathcal B(c|d)=\{U\subset X: U\hbox{ is an open neighborhood of } c\hbox{ and }d\notin \overline{U}\}.$$
  Put
$\mathcal{B}(a)=\mathcal B(a|b)$ and $\mathcal{B}(b)=\mathcal B(b|a)$.
 Let $Z$ be the Tychonoff product $X{\times}\omega$ where $\omega$ is endowed with the discrete topology. For any $x\in X$ and $n\in\omega$ by $x_n$ we denote the point $(x,n)$. Analogously, for any $B\subset X$ and $n\in\omega$ the set $B{\times}\{n\}$ is denoted by $B_n$. Let $D$ be a dense countable subset of $X$ such that $D\subset X\setminus\{a,b\}$. Put $E=\cup_{n\in\omega}D_{n}$. Fix any bijection $f:E\rightarrow \omega$ such that $f(x_{n})\neq n$ for each $n\in\omega$ and $x_{n}\in D_{n}$.  On the set $Z$ consider the smallest equivalence relation $\sim$ which satisfies the following conditions:
\begin{itemize}
\item $x_{n}\sim a_{f(x_n)}$ for any $n\in\omega$ and $x\in D$;
\item $b_{n}\sim b_{m}$ for any $n,m\in\omega$.
\end{itemize}

Let $D_{a,b}(X)$ be the quotient space $Z/_{\sim}$. By $\pi$ we denote the natural projection of $Z$ onto $D_{a,b}(X)$. Put $b^*=\pi(b_{n})\in D_{a,b}(X)$. Observe that the definition of the element $b^*$ does not depend on the choice of $n\in\omega$. A subset $B\subset Z$ is called {\em saturated} if $\pi^{-1}(\pi(B))=B$.
The following straightforward observation might be useful for arguments appearing later:
For any  $Q\subset X$ not containing $a,b$, the saturation $\pi^{-1}(\pi(Q_n))$ is the union of $Q_n$ with the closed discrete
set  $\{a_k:k\in f(Q_n\cap D_n)\}$.

\begin{lemma}\label{l0}
Let $X$ be a regular separable space. Then the space $D_{a,b}(X)$ is regular, separable and any continuous map $g:D_{a,b}(X)\rightarrow Y\in \operatorname{Const}(X)_{a,b}$ is constant.
\end{lemma}

\begin{proof}
Observe that the countable set $\pi (E)$ is dense in $D_{a,b}(X)$ witnessing that the space $D_{a,b}(X)$ is separable.
Since each equivalence class of the relation $\sim$ is closed in $Z$ the space $D_{a,b}(X)$ is $T_1$. Fix any $q\in D_{a,b}(X)\setminus\{b^*\}$ and an open neighborhood $U$ of $q$. It is easy to see that $\pi^{-1}(q)\setminus\{a_k\mid k\in\omega\}$ is a singleton. It follows that there exists $n\in\omega$ such that $\pi^{-1}(q)\setminus\{a_k\mid k\in\omega\}=\{x_n\}$. By the definition of the quotient topology, to prove the regularity of $D_{a,b}(X)$ at the point $q$ it is sufficient to find an open saturated subset $B\subset Z$ and a closed saturated subset $C\subset Z$ such that $x_n\in B\subset C\subset \pi^{-1}(U)$ (note that if $x_n\in B$, then $\pi^{-1}(q)\subset B$). The subsets $B$ and $C$ will be constructed inductively. Since $X_n$ is a clopen subspace of the regular space $Z$ and $x_n\notin\{a_n, b_n\}$ there exists an open subset $V\subset X_n$ such that $a_n, b_n\notin \overline{V}$ and $x_n\in V\subset \overline{V}\subset \pi^{-1}(U)$. Put $B_{(0)}=C_{(0)}=\emptyset$, $B_{(1)}=V$ and $C_{(1)}=\overline{V}$.  Assume that for every $n\leq k$ we already constructed subsets $B_{(n)}$ and $C_{(n)}$ such that $B_{(n)}$ is open, $C_{(n)}$ is closed, $B_{(n)}\subset C_{(n)}\subset \pi^{-1}(U)$ and $B_{(i)}\subset B_{(i+1)}$, $C_{(i)}\subset C_{(i+1)}$ for every $i<k$. Put $B_{(k)}^*=(B_{(k)}\setminus B_{(k-1)})\cap E$ and $C_{(k)}^*=(C_{(k)}\setminus C_{(k-1)})\cap E$. Fix any function $t:C_{(k)}^*\rightarrow \mathcal{B}(a)$ such that $\overline{t(y)}_{f(y)}\subset \pi^{-1}(U)$ for each $y\in C_{(k)}^*$. Since $\pi^{-1}(U)$ is an open saturated subset and the space $Z$ is regular such a function exists.
Put $s=t {\uhr} B_{(k)}^*$.
Let $B_{(k+1)}=\cup\{s(y)_{f(y)}\mid y\in B_{(k)}^*\}\cup B_{(k)}$ and $C_{(k+1)}=\cup\{\overline{t(y)}_{f(y)}\mid y\in C_{(k)}^*\}\cup C_{(k)}$. Finally, put $B=\cup_{n\in\omega}B_{(n)}$ and $C=\cup_{n\in\omega}C_{(n)}$. At this point it is straightforward to check that the set $B$ ($C$, resp.) is an open (closed, resp.) saturated subset of $Z$ such that $\pi^{-1}(q)\subset B\subset C\subset \pi^{-1}(U)$. Hence the space $D_{a,b}(X)$ is regular at any point $q\neq b^*$.

Fix any open neighborhood $U$ of $b^*$ in $D_{a,b}(X)$. By the regularity of $Z$, for each $n\in\omega$ there exists an open neighborhood $V_n\subset X_n$ of $b_n$ such that $V_n\subset \overline{V_n}\subset \pi^{-1}(U)$. Put $B_{(0)}=C_{(0)}=\emptyset$ $B_{(1)}=\cup_{n\in\omega}V_n$ and $C_{(1)}=\cup_{n\in\omega}\overline{V_n}$. Observe that the family $\{\overline{V_n}\}_{n\in\omega}$ is locally finite which yields that the set $C_{(1)}$ is closed. Next, repeating the previous inductive construction we can find a desired sets $B$ and $C$ such that $B$ is open saturated, $C$ is closed saturated and $\pi^{-1}(b^*)\subset B\subset C\subset \pi^{-1}(U)$. Hence the space $D_{a,b}(X)$ is regular.

Fix any continuous map $g:D_{a,b}(X)\rightarrow Y\in\operatorname{Const}(X)_{a,b}$. Note that for any $n\in \omega$, $x\in X_n$ and an open neighborhood $U\subset X_n$ of $x$ there exists an open saturated subset $W\subset Z$ such that $U\subset W$ and $W\cap X_n=U$. Thus the map $h:X\rightarrow\pi(X_{n})$, $h(x)=\pi(x_{n})$ defines a homeomorphism between $X$ and the subspace $\pi(X_{n})$ of $D_{a,b}(X)$. Hence for any $n\in\omega$, $g(\pi(a_{n}))=g(\pi(b_{n}))$ by the choice of $Y$. The definition of the equivalence relation $\sim$ implies that $g(\pi(a_{n}))=g(b^*)$ for any $n\in\omega$. Since the set $\{\pi(a_{n})\mid n\in\omega\}=\pi(E)$ is dense in $D_{a,b}(X)$ and $Y$ is a $T_1$ space, the map $g$ is constant on the whole $D_{a,b}(X)$.
\end{proof}

The following lemmas will be useful in the next sections.
\begin{lemma}\label{l2}
Let $X$ be a regular separable countably compact space. Then a subset $P$ of $D_{a,b}(X)$ is closed and discrete iff $|P\cap\pi(X_n)|<\omega$ for each $n\in\omega$.
\end{lemma}

\begin{proof}
Let $P$ be a closed discrete subset of $D_{a,b}(X)$. Recall that the subspace $\pi(X_n)\subset D_{a,b}(X)$ is homeomorphic to $X$. Since $X$ is countably compact and $P$ is closed and discrete, the set $P\cap\pi(X_n)$ is finite for each $n\in\omega$.

Let $P$ be any subset of $D_{a,b}(X)$ such that $|P\cap\pi(X_n)|<\omega$ for each $n\in\omega$.
Observe that for any $x\in P$ the set $U=Z\setminus \pi^{-1}(P\setminus\{x\})$ is an open saturated subset of $Z$ which contains $\pi^{-1}(x)$. Thus $\pi(U)$ is an open neighborhood of $x$ in $D_{a,b}(X)$ such that $\pi(U)\cap P=\{x\}$.
Hence the set $P$ is discrete.

It is straightforward to check that the set $\pi^{-1}(P)$ is closed in $Z$. By the definition of the quotient topology, the set $P$ is closed in $D_{a,b}(X)$.
\end{proof}

 A family $\mathcal I\subset [X]^{\leq \omega}$ is called {\em P-ideal}, if for every $\mathcal J\in [\mathcal I]^{\omega}$ there exists $I\in\mathcal I$ such that $J\subset^* I$ for all $J\in\mathcal J$. In this paper ideals are not assumed to be proper, i.e., it is possible that $X\in \mathcal I$. This is justified by the following lemma, where $X$ can be a countable discrete space.

\begin{corollary}\label{c1}
 Let $X$ be a regular separable countably compact space. Then the family $\mathcal{F}$ of all countable closed discrete subsets of $D_{a,b}(X)$ is a P-ideal.
\end{corollary}

\begin{proof}
Fix any countable family $\mathcal{F}=\{P_{(n)}\}_{n\in\omega}$ of countable closed discrete subsets of $D_{a,b}(X)$. By Lemma~\ref{l2}, the set $P_{(n)}\cap \pi(X_k)$ is finite for any $n,k\in\omega$. Put $P_{(n)}^*=P_{(n)}\setminus (\cup_{k\leq n}\pi(X_k))$, $n\in\omega$. Let $P=\cup_{n\in\omega}P_{(n)}^*$. Observe that $|P\cap \pi(X_n)|<\omega$ for each $n\in\omega$. Lemma~\ref{l2} implies that $P$ is closed and discrete in $D_{a,b}(X)$. Since $P_{(n)}\subset^* P_{(n)}^*$, we have that $P_{(n)}\subset^* P$ for each $n\in\omega$.
\end{proof}

Following~\cite{BBR}, a subset $P$ of a space $X$ is called {\em strongly discrete} if each point $x\in P$ has a neighborhood $O_x\subset X$ such that the family $\{O_x\mid x\in P\}$ is disjoint and locally finite in $X$. A space $X$ is said to have {\em property $D$} if any countable closed discrete subset of $X$ is strongly discrete.

\begin{lemma}\label{l1}
Let $X$ be a regular separable countably compact space. Then $D_{a,b}(X)$ has property $D$.
\end{lemma}

\begin{proof}
Fix any countable closed discrete subset $P$ of $D_{a,b}(X)$. With no loss of generality we can assume that $b^*\notin P$. Lemma~\ref{l2} implies that for each $n\in\omega$ the set $\pi^{-1}(P)\cap X_n$ is finite. It follows that $\pi^{-1}(P)$ is a closed discrete subset of $Z$.
For each $p\in P$ there exists $x(p)\in X$ and $n(p)\in\omega$ such that $\{x(p)_{n(p)}\}=\pi^{-1}(p)\setminus\{a_k\mid k\in\omega\}$. Fix an open subset $V(p)\subset X_{n(p)}$ which satisfies the following conditions:
\begin{itemize}
\item[$(i)$] $x(p)_{n(p)}\in V(p)$;
\item[$(ii)$] $V(p)\cap V(q)=\emptyset$ if $p\neq q$;
\item[$(iii)$] $a_{n(p)}\notin \overline{V(p)}$;
\item[$(iv)$] $b_{n(p)}\notin \overline{V(p)}$.
\end{itemize}
Next we shall inductively construct open saturated pairwise disjoint subsets $W(p)\subset Z$, $p\in P$ such that $\pi^{-1}(p)\subset W(p)$ and the family $\{\pi (W(p))\mid p\in P\}$ is locally finite. For each $p\in P$ put $W(p)_{(0)}=\emptyset$ and $W(p)_{(1)}=V(p)$. Assume that we already constructed subsets $W(p)_{(n)}$ for each $n\leq k$ and $p\in P$  such that $W(p)_{(n)}\cap W(q)_{(n)}=\emptyset$ for any distinct $p,q$ from $P$.  Put $W(p)_{(k)}^*=(W(p)_{(k)}\setminus W(p)_{(k-1)})\cap E$ and fix any function $g:W(p)_{(k)}^*\rightarrow \mathcal{B}(a)$ which satisfies the following condition:
\begin{itemize}
\item[$(v)$] $g(z)_{f(z)}\cap V(q) =\emptyset$ for every $z\in W(p)_{(k)}^*$ and $q\in P$.
\end{itemize}
Such a function exists by the choice of $V(p)$, $p\in P$ (see condition $(iii)$ above).  Let $W(p)_{(k+1)}=\cup\{g(z)_{f(z)}\mid z\in W(p)_{(k)}^*\}\cup W(p)_{(k)}$. Finally, for each $p\in P$ put $W(p)=\cup_{n\in\omega}W(p)_{(n)}$. It is easy to see that subsets $W(p)$, $p\in P$ are open and saturated. Since the map $f$ is bijective the sets $W(p)$, $p\in P$ are pairwise disjoint.

It remains to show that the family $\{\pi(W(p))\mid p\in P\}$ is locally finite. For this fix any $n\in \omega$ and $x_n\in X_n\setminus\{a_n\}$.  Next we shall inductively construct an open saturated subset $U\subset Z$ which contains $x_n$ and intersects only finitely many members of the family $\{W(p)\mid p\in P\}$. If $x_n\in W(p)$ for some $p\in P$, then put $U=W(p)$. If $x_n\notin \cup\{W(p)\mid p\in P\}$ there are two cases to consider:

1. $x_n\neq b_n$.
Put $U_{(0)}=\emptyset$ and  $U_{(1)}=X_n\setminus\{b_n\}$. Assume that we already constructed $U_{(n)}$ for each $n\leq k$. Let $U_{(k)}^{*}=(U_{(k)}\setminus U_{(k-1)})\cap E$. Consider any function $g:U_{(k)}^*\rightarrow \mathcal{B}(a)$ such that $g(z)_{f(z)}\cap V(p)=\emptyset$ for any $p\in P$ and $z\in U_{(k)}^*$. The existence of such a function is provided by the bijectivity of $f$ together with the choice of the sets $V(p)$ (see condition $(iii)$ above). Put $U_{(k+1)}=\cup\{g(y)_{f(y)}\mid y\in U_{(k)}^*\}\cup U_{(k)}$. Finally, let $U=\cup_{n\in\omega}U_{(n)}$. Observe that there exist only finitely many $p\in P$ such that $W(p)\cap X_n\neq \emptyset$. It is easy to check that $U$ is an open saturated subset of $Z$ which contains $x_n$ and intersects only finitely many elements of the set $\{W(p)\mid p\in P\}$. Namely, $U\cap W(p)\neq\emptyset$ iff $W(p)\cap X_n\neq \emptyset$.

2. $x_n=b_n$. The choice of the sets $V(p)$ (see condition $(iv)$ above) and the definition of the sets $W(p)$ imply that for each $k\in\omega$ there exists an open neighborhood $V_k$ of $b_k$ such that $\cup\{W(p)\mid p\in P\}\cap V_k=\emptyset$.  Put $U_{(0)}=\emptyset$, $U_{(1)}=\cup_{k\in\omega}V_k$ and repeat the previous inductive construction to obtain an open saturated subset $U\subset Z$ such that $\{b_k\mid k\in\omega\}\subset U$ and $U\cap W(p)=\emptyset$ for all $p\in P$.

Hence for any $x\in Z$ there exists an open saturated set $U$ such that $x\in U$ and $\pi(U)\cap \pi(W(p))=\emptyset$ for all but finitely many $p\in P$ witnessing that the family $\{\pi(W(p))\mid p\in P\}$ is locally finite.
\end{proof}

Next we briefly recall a Jones machine (see~\cite{J} for more details) which turn a regular non-normal space $X$ into a regular space $J(X)$ which contains points $a,b$ such that $\mathbb{R}\in \operatorname{Const}(J(X))_{a,b}$. So, let $X$ be a regular space containing two closed sets $A$ and $B$ which cannot be separated by disjoint open sets. Let $\mathbb{Z}$ be a discrete space of integers and $a,b$ be distinct points which do not belong to $X{\times}\mathbb Z$. By $R$ we denote the set $X{\times}\mathbb Z\cup\{a,b\}$ endowed with the topology $\tau$ which satisfies the following conditions:
\begin{itemize}
\item the Tychonoff product $X{\times}\mathbb Z$ is open in $R$;
\item if $a\in U\in\tau$, then there exists $n\in\mathbb N$ such that $\{(x,-k)\mid x\in X$ and $k>n\}\subset U$;
\item if $b\in U\in\tau$, then there exists $n\in\mathbb N$ such that $\{(x,k)\mid x\in X$ and $k>n\}\subset U$.
\end{itemize}

On the space $R$ consider the smallest equivalence relation $\sim$ such that
$$(x,2n)\sim (x,2n+1)\hbox{ and } (y,2n)\sim(y,2n-1)$$
for any $n\in \mathbb Z$, $x\in A$ and $y\in B$.

Let $J(X)$ be the quotient space $R/_{\sim}$.
Note that $J(X)$ contains a homeomorphic copy of $X$ (for instance, one can consider the subspace $\pi (X\times\{n\})\subset J(X)$ for any integer $n$).
We claim that $f(a)=f(b)$ for each continuous map $f:J(X)\rightarrow \mathbb{R}$. Indeed, if $f(a)\neq f(b)$, then fix a sequence $\{U_{(n)}\mid n\in\omega\}$ of open neighborhoods of $f(a)$ such that $f(b)\notin \overline{\cup_{n\in\omega}U_{(n)}}$ and $\overline{U_{(n)}}\subset U_{(n+1)}$ for each $n\in\omega$. Let $k=\max\{n\in \mathbb{Z}\mid \pi(A{\times}\{n\})\subset f^{-1}(U_0)\}$. Observe that $B_1=\overline{f^{-1}(U_0)}\cap \pi(B{\times}\{k+1\})$ is a nonempty subset of $f^{-1}(U_{(1)})$ witnessing that $f^{-1}(U_{(1)})\cap \pi(X{\times}(k+1))\neq \emptyset$. It is easy to check that the closed sets $B_1$ and $\pi(A{\times}(k+2))$ cannot be separated by disjoint open subsets in $J(X)$. Thus $A_1=\overline{f^{-1}(U_1)}\cap \pi(A{\times}\{k+2\})$ is a nonempty subset of $f^{-1}(U_{(2)})$ and therefore $f^{-1}(U_{(2)})\cap \pi(X{\times}(k+2))\neq \emptyset$.
By the induction on $n\in\omega$ it can be shown that $\pi(X{\times}\{k+n\})\cap f^{-1}(U_{(n)})\neq \emptyset$ for each $n\in\omega$. Since $f(b)\notin \overline{\cup_{n\in\omega}U_{(n)}}$ there exists an open neighborhood $V$ of $f(b)$ such that $(\cup_{n\in\omega}U_{(n)})\cap V=\emptyset$. By the definition of the topology on $J(X)$, there exists  $m>k$ such that $\pi(X{\times}\{i\})\subset f^{-1}(V)$ for each $i\geq m$. However, $f^{-1}(V)\cap f^{-1}(U_{(k+m)})\neq \emptyset$  which implies a contradiction. Hence $f(a)=f(b)$ for each continuous real-valued function on $J(X)$ or, equivalently, $\mathbb{R}\in\operatorname{Const}_{a,b}(J(X))$.

Note that the space $J(X)$ preserves many different properties of a space $X$, in particular, regularity, separability, first countability and countable compactness.


\section{Separable regular totally countably compact $\R$-rigid spaces in ZFC}

A space $X$ is called {\em totally countably compact} if each infinite set in $X$ contains an infinite subset with compact closure in $X$.
A family $\mathcal{F}$ of closed subsets of a topological space $X$ is called a {\em closed ultrafilter} if it satisfies the following conditions:
\begin{itemize}
\item $\emptyset\notin\mathcal{F}$;
\item $A\cap B\in\mathcal{F}$ for any $A,B\in\mathcal{F}$;
\item a closed set $F\subset X$ belongs to $\mathcal F$ if $F\cap A\ne\emptyset$ for every $A\in\mathcal{F}$.
\end{itemize}
 For any $A\subset X$ let $\langle A\rangle$ be the set of all closed ultrafilters which contain an element $F\subset A$.
Following~\cite[\S3.6]{Eng}, the Wallman extension $W(X)$ of a $T_1$ space $X$ consists of closed ultrafilters on $X$ and carries the topology generated by the base consisting of the sets $\langle U\rangle$ where $U$ runs through open subsets of $X$.

By \cite[Theorem~3.6.21]{Eng}, the Wallman extension $W(X)$ is $T_1$ and compact. A $T_1$ space $X$ is normal if and only if $W(X)$ is Hausdorff~\cite[3.6.22]{Eng}.
Consider the map $j_X:X\to W(X)$ assigning to each point $x\in X$ the principal closed ultrafilter consisting of all closed sets $F\subset X$ containing the point $x$. It is clear that the image $j_X(X)$ is dense in $W(X)$. Since the space $X$ is $T_1$, Theorem~3.6.21 from~\cite{Eng} provides that the map $j_X:X\to W(X)$ is a topological embedding. So, $X$ can be identified with the subspace $j_X(X)$ of $W(X)$.
Before formulating the main result of this section we need a few auxiliary lemmas.

\begin{lemma}[{\cite[Lemma 1]{BO}}]\label{W}
Let $\mathcal{F}\in W(X)$ and $A\subset X$. Then $\mathcal{F}\in \cl_{W(X)}(A)$ if and only if $\cl_X(A)\in\mathcal{F}$.
\end{lemma}

\begin{corollary}\label{Wc}
For any subset $V\subset X$, $\langle \cl_X(V)\rangle=\cl_{W(X)}(V)=\cl_{W(X)}(\langle V\rangle)$.
\end{corollary}

\begin{proof}
Let $\mathcal F$ be a closed ultrafilter on $X$. Assume that $\cl_X(V)\in \mathcal F$. Fix any basic open neighborhood $\langle U\rangle$ of $\mathcal F$. Then the open set $U$ contains some element $G\in \mathcal F$ witnessing that $\emptyset\neq\cl_X(V)\cap G\subset U$. Fix any $y\in \cl_X(V)\cap G$. Since $U$ is an open neighborhood of $y$ we get that $U\cap V\neq \emptyset$. Hence $\mathcal F\in \cl_{W(X)}(V)\subset \cl_{W(X)}(\langle V\rangle)$.

Assume that $\cl_X(V)\notin \mathcal F$. Since $\mathcal F$ is a closed ultrafilter there exists $G\in\mathcal F$ such that $G\cap\cl_X(V)=\emptyset$.
Then $\langle X\setminus \cl_X(V)\rangle$ is an open neighborhood of $\mathcal F$ disjoint with $V$, and hence also with $\langle V\rangle$. Thus $\langle \cl_X(V)\rangle=\cl_{W(X)}(V)=\cl_{W(X)}(\langle V\rangle)$.
\end{proof}

For a topological space $X$ by $cW(X)$ we denote the following subspace of the Wallman extension $W(X)$:
$$cW(X)=\cup\{\cl_{W(X)}(A)\mid A \hbox{ is a countable discrete closed subset of } X\}.$$

\begin{lemma}\label{l3}
 Let $X$ be a regular space which has property $D$ and the family of all countable closed discrete subsets of $X$ forms a P-ideal.
Then the space $cW(X)$ is regular and countably compact.
\end{lemma}

\begin{proof}
Since the Wallman extension $W(X)$ is $T_1$, the space $cW(X)$ is $T_1$ as well. To prove the regularity of $cW(X)$ we will show that $W(X)$ is regular at any point of $cW(X)$. Fix any $x\in X$ and open neighborhood $\langle U\rangle$ of $x$ in $W(X)$. By the regularity of $X$, there exists an open neighborhood $V\subset X$ of $x$ such that $\cl_{X}(V)\subset U$. Corollary~\ref{Wc} provides that
$\langle V\rangle\subset\langle \cl_{X}(V)\rangle=\cl_{W(X)}(\langle V\rangle)\subset \langle U\rangle$. Hence the space $W(X)$ is regular at any point $x\in X$.

Consider any closed ultrafilter $\mathcal{F}\in cW(X)\setminus X$ and fix any basic open neighborhood $\langle U\rangle$ of $\mathcal{F}$ in $W(X)$. Take any $F\in\mathcal{F}$ such that $F\subset U$. By the definition of $cW(X)$, there exists a countable closed discrete subset $G\subset X$ such that $\mathcal{F}\in \cl_{W(X)}(G)$. Lemma~\ref{W} implies that $G\in \mathcal{F}$. Put $H=F\cap G$ and note that $H$ is closed and discrete, $H\subset U$ and $H\in\mathcal{F}$.
Since the space $X$ possesses property $D$, the subset $H$ is strongly discrete. Hence there exists a locally finite family $\{W(y)\mid y\in H\}$ of pairwise disjoint open subsets of $X$ such that $y\in W(y)$ for each $y\in H$. Also, with no loss of generality we can assume that $\cup\{W(y)\mid y\in H\}\subset U$. By the regularity of $X$, for every $y\in H$ there exists an open subset $V(y)$ of $X$ such that $y\in V(y)\subset \cl_{X}(V(y))\subset W(y)$. Since the family $\{W(y)\mid y\in H\}$ is locally finite the set $\cup\{\cl_{X}(V(y))\mid y\in H\}$ is closed in $X$. Then Corollary~\ref{Wc} implies the following:
$$\mathcal{F}\in\langle\cup\{V(y)\mid y\in H\}\rangle\subset \langle\cup\{\cl_{X}(V(y))\mid y\in H\}\rangle=\cl_{W(X)}(\langle\cup\{V(y)\mid y\in H\}\rangle)\subset \langle U\rangle.$$
Hence the space $W(X)$ is regular at any point of $cW(X)$ witnessing that $cW(X)$ is regular.

It remains to show that $cW(X)$ is countably compact. Assuming the contrary, let $A$ be an infinite closed discrete subset of $cW(X)$. The definition of the space $cW(X)$ implies that the set $A\cap X$ is finite.
 Consider any countable infinite subset $D=\{\mathcal{F}_{(n)}\mid n\in\omega\}\subset A\cap (cW(X)\setminus X)$. For each $n\in\omega$ there exists a countable closed discrete subset $C_{(n)}\subset X$ such that $C_{(n)}\in\mathcal{F}_{(n)}$. By the assumption, there exists a countable closed discrete subset $C\subset X$  such that $C_{(n)}\subset^* C$. Note that $C\in\mathcal{F}_{(n)}$ for each $n\in\omega$. Then $\{\mathcal{F}_{(n)}\mid n\in\omega\}$ is contained in a compact subset $\cl_{cW(X)}(C)=\cl_{W(X)}(C)$ which implies a contradiction.
Hence the space $cW(X)$ is countably compact.
\end{proof}

Now we are able to formulate the main result of this section which gives an affirmative answer to Problem~\ref{p1} for separable spaces.

\begin{theorem}\label{t1}
Each separable regular non-normal (totally) countably compact space $X$  can be embedded into a regular separable (totally) countably compact $\R$-rigid space $R(X)$.
\end{theorem}

\begin{proof}
Let $X$ be any separable regular non-normal (totally) countably compact space. We shall make three steps to construct a desired space $R(X)$.

{\bf Step 1}. Using Jones machine (see the end of Section~\ref{V}), construct the space $J(X)$. It is straightforward to check that $J(X)$ remains regular, separable, (totally) countably compact, and contains a homeomorphic copy of $X$. Also, $f(a)=f(b)$ for each continuous real-valued function on $J(X)$.

{\bf Step 2}. Consider the space $D_{a,b}(J(X))$. By Lemma~\ref{l0}, the space $D_{a,b}(J(X))$ is separable, re\-gu\-lar and admits only constant continuous real-valued functions. Also, recall that $D_{a,b}(J(X))$ contains a homeomorphic copy of $J(X)$ and, as a consequence, of $X$.  However, the space $D_{a,b}(J(X))$ fails to be countably compact.

{\bf Step 3}. Consider the space $R(X)=cW(D_{a,b}(J(X)))$. By Lemma~\ref{l1} and Corollary~\ref{c1}, the space $D_{a,b}(J(X))$ satisfies conditions of Lemma~\ref{l3}. Hence the space $R(X)$ is regular separable countably compact and contains $D_{a,b}(J(X))$ as a dense subspace which implies that $R(X)$ is $\R$-rigid.

It remains to check that the space $R(X)$ is totally countably compact whenever $X$ is totally countably compact. Recall that in this case $J(X)$ is totally countably compact. Fix any countable infinite subset $D\subset R(X)$. There are two cases to consider:

 1. The set $D\cap D_{a,b}(J(X))$ is infinite.
If there exists $n\in\omega$ such that the set $M=\pi(J(X)_n)\cap D$ is infinite, then by the total countable compactness of $J(X)$ (recall that $\pi(J(X)_n)$ is homeomorphic to $J(X)$) there exists an infinite subset $N\subset M$ such that $\overline{N}$ is compact in $R(X)$. If for each $n\in \omega$ the set $D\cap \pi(J(X)_n)$ is finite, then Lemma~\ref{l2} implies that the set $D$ is closed and discrete in $D_{a,b}(J(X))$. Hence the set $\cl_{R(X)}(D)=\cl_{W(D_{a,b}(J(X)))}(D)$ is compact.

2. The set $D\cap D_{a,b}(J(X))$ is finite. Let $D\cap (R(X)\setminus D_{a,b}(J(X)))=\{\mathcal{F}_{(n)}\mid n\in\omega\}$. For each $n\in\omega$ there exists a closed discrete subset $C_{(n)}\subset D_{a,b}(J(X))$ such that $C_{(n)}\in\mathcal{F}_{(n)}$. By Corollary~\ref{c1}, there exists a countable closed discrete subset $C\subset D_{a,b}(J(X))$ such that $C_{(n)}\subset^* C$. Note that $C\in\mathcal{F}_{(n)}$ for each $n\in\omega$. Then $\{\mathcal{F}_{(n)}\mid n\in\omega\}$ is contained in a compact subset $\cl_{R(X)}(C)=\cl_{W(D_{a,b}(J(X)))}(C)$.
Hence the space $R(X)$ is totally countably compact.
\end{proof}

The following example shows that the spaces used in the statement of Theorem~\ref{t1} exist in ZFC.

\begin{example} Fix any ultrafilter $p\in\omega^*$ such that the subspace $\beta(\omega)\setminus \{p\}$ is non-normal. The existence in ZFC of such an ultrafilter was proved by Blaszczyk and Szyma\'{n}ski in~\cite{BS}. We claim that the space $\beta(\omega)\setminus\{p\}$ is separable, regular, non-normal and totally countably compact. The first three conditions are obvious. Fix any countable subset $A\subset \beta(\omega)\setminus\{p\}$. Since $\beta(\omega)$ contains only trivial convergent sequences there exists an open neighborhood $U$ of $p$ in $\beta(\omega)$ such that the set $B=A\setminus U$ is infinite. Then $\overline{B}$ is a compact subset of $\beta(\omega)\setminus\{p\}$ witnessing that the space $\beta(\omega)\setminus\{p\}$ is totally countably compact.
\end{example}



\section{Consistent generalization}
For a cardinal $\theta$ a subset $\Tau=\{T_{\alpha}\}_{\alpha\in\theta}\subset [\omega]^{\omega}$ is called a {\em tower} if for any $\alpha<\beta<\theta$, $T_{\alpha}\subset^*T_{\beta}$ and $|T_{\beta}\setminus T_{\alpha}|=\omega$. A tower $\Tau=\{T_{\alpha}\}_{\alpha\in\theta}$ is called {\em regular} if the cardinal $\theta$ is regular. A tower $\Tau$ is called {\em maximal} if for any infinite subset $I$ of $\omega$ there exists $T_{\alpha}\in \Tau$ such that the set $I\cap T_{\alpha}$ is infinite.

In this section, we provide a consistent generalization of Theorem~\ref{t1}. For this purpose we will essentially use an example of Franklin and Rajagopalan constructed in~\cite{FR} (see also Example 7.1 in~\cite{int}).

For any ordinals $\alpha,\beta$ by $[\alpha,\beta]$ (resp. $[\alpha,\beta)$) we denote the set of all
ordinals $\gamma$ such that $\alpha\leq \gamma\leq \beta$ (resp. $\alpha\leq \gamma< \beta$). Further, if some ordinal $\alpha$ is considered as a topological space, then $\alpha$ is assumed to carry the order topology.

\begin{lemma}[{\cite[Lemma 2]{BO}}]\label{CCC}
Let $\kappa$ be an infinite cardinal, $\xi$ a regular cardinal $> \kappa^+$ and $Y$
a $T_1$ space of pseudocharacter $\leq \kappa$. Then for each continuous map
$f : \xi \rightarrow Y$ there exist $y\in Y$ and $\mu\in \xi$ such that $[\mu, \xi) \subset f^{-1}(y)$.
\end{lemma}

For any cardinals $\alpha,\beta$, by $T_{\alpha,\beta}$ we denote the subspace $(\alpha+1){\times}(\beta+1)\setminus\{(\alpha,\beta)\}$ of the Tychonoff product $(\alpha+1){\times}(\beta+1)$.

\begin{lemma}\label{CC}
Let $\kappa$ be an infinite cardinal, $\xi$ a regular cardinal $> \kappa^+$ and $Y$
a $T_1$ space of pseudocharacter $\leq \kappa$. Then for each continuous map
$f : T_{\xi,\xi} \rightarrow Y$ there exist $y\in Y$ and $\mu\in \xi$ such that $[\mu, \xi]^2\setminus (\xi,\xi) \subset f^{-1}(y)$.
\end{lemma}

\begin{proof}
Fix any $T_1$ space $Y$ of pseudocharacter $\leq \kappa$ and a continuous map $f:T_{\xi,\xi} \rightarrow Y$. Note that the diagonal $\{(\alpha,\alpha): \alpha\in \xi\}$ of $T_{\xi,\xi}$ is homeomorphic to $\xi$. Then Lemma~\ref{CCC} implies that there exist $y\in Y$ and $\delta\in \xi$ such that $\{(\alpha,\alpha):\alpha\in [\delta,\xi)\}\subset f^{-1}(y)$. Since pseudocharacter of $Y$ is $\leq \kappa$ there exists a family $\{U_{\beta}:\beta\in \kappa\}$ of open neighborhoods of $y$ such that $\cap_{\beta\in\kappa}U_{\beta}=\{y\}$. Fix any $\beta\in\kappa$. Since $f^{-1}(U_{\beta})$ is open, for each $\alpha\in[\delta,\xi)$ there exists $\theta(\alpha)<\alpha$ such that $(\theta(\alpha),\alpha]^2 \subset f^{-1}(U_{\beta})$. Then the defined above map $\theta: [\delta,\xi)\rightarrow \xi$ satisfies the conditions of Fodor's Lemma~\cite[Lemma III.6.14]{Kun}. Hence there exists $\eta_\beta\in \xi$ such that for stationary many $\alpha\in [\delta, \xi)$ we have that $\theta(\alpha)=\eta_{\beta}$. It is easy to check that  $[\eta_{\beta},\xi)^2 \subset f^{-1}(U_{\beta})$.
Let $\mu=\sup_{\beta\in\kappa}\eta_{\beta}$. Since $cf(\xi)>\kappa$ we get that $\mu\in\xi$.
Observe that
$$[\mu, \xi)^2\subset \cap_{\beta\in\kappa} f^{-1}(U_{\beta})=f^{-1}(\cap_{\beta\in\kappa}U_{\beta})=f^{-1}(y).$$
Since the space $Y$ is $T_1$ we get that the set $f^{-1}(y)$ is closed. Therefore
$$[\mu, \xi]^2\setminus (\xi,\xi)=\cl_{T_{\xi,\xi}}([\mu, \xi)^2)\subset f^{-1}(y).$$
\end{proof}

The following theorem is the main result of this section.

\begin{theorem}\label{t2}
Let $\kappa$ be an infinite cardinal such that there exists a regular maximal tower $\Tau$ satisfying $\kappa^+<|\Tau|$. Then there exists a separable totally countably compact space $Q$ such that $Q$ is $Y$-rigid for any $T_1$ space $Y$ of pseudocharacter $\leq \kappa$.
\end{theorem}
\begin{proof}
Fix an infinite cardinal $\kappa$ and a regular maximal tower $\Tau=\{T_{\alpha}\mid \alpha\in \lambda\}$ where $\kappa^+<\lambda$.
Consider the space $Y=\Tau\cup\w$ which is topologized as follows. Points of $\omega$ are isolated and basic open neighborhoods of $T\in \Tau$ have the form
$$B(S,T,F)=\{P\in\Tau\mid S\subset^* P\subseteq^* T\}\cup ((T\setminus S)\setminus F),$$
where $S\in \Tau\cup\{\emptyset\}$ satisfies $S\subset^{*}T$ and $F$ is a finite subset of $\omega$.

Repeating arguments used in~\cite[Example 7.1]{int} one can check that the space $Y$ is sequentially compact, separable, normal and locally compact.
Choose any point $\infty\notin Y$ and let $X=\{\infty\}\cup Y$ be the one-point compactification of the locally compact space $Y$.
Let $\Pi=X^2\setminus\{(\infty,\infty)\}$. It is easy to see that the space $\Pi$ is regular separable and sequentially compact. Consider the closed subsets $A=\{\infty\}{\times} \Tau$ and $B=\Tau{\times}\{\infty\}$. Note that the subspace $\Tau{\times}\Tau\cup A\cup B\subset \Pi$ is homeomorphic to the space $T_{\lambda,\lambda}$, defined just before Lemma~\ref{CC}.  Let us mention that the sets $A$ and $B$ can be separated by disjoint open neighborhoods. Nonetheless,
using Jones machine, we construct the space $J(\Pi)$ precisely as we did at the end of Section~\ref{V}.
 This time we shall use another argument to show that $f(a)=f(b)$ for any continuous map $f$ from $J(\Pi)$ into
 appropriate spaces.

It is easy to check that $J(\Pi)$ is separable regular and sequentially compact space.
Fix any $T_1$ space $Y$ with pseudocharacter $\leq \kappa$ and a continuous map $f:J(\Pi)\rightarrow Y$. We claim that $f(a)=f(b)$.  Recall that in the construction of $J(\Pi)$ by $R$ we denoted the space $\Pi{\times}\mathbb Z\cup\{a,b\}$ and by $\pi$ the quotient map from $R$ onto $J(\Pi)$.  Consider the map $f\circ \pi: R\rightarrow Y$. Note that for each $n\in\mathbb{Z}$ the subspace $(\Tau{\times}\Tau\cup A\cup B){\times} \{n\}$ is homeomorphic to $T_{\lambda,\lambda}$. Thus, Lemma~\ref{CC} implies that for each $n\in \mathbb Z$ there exist $y_n\in Y$ and ordinal $\alpha_n\in\lambda$ such that
$$\{(x,\infty,n)\mid x\in \Tau\setminus\{T_{\xi}\mid \xi\leq \alpha_n\}\}\cup \{(\infty,x,n)\mid x\in \Tau\setminus\{T_{\xi}\mid \xi\leq \alpha_n\}\}\subset \pi^{-1}(f^{-1}(y_n)).$$
Put $\alpha=\sup\{\alpha_n\mid n\in \mathbb Z\}$.
Since $\pi((\infty,x,2n))=\pi((\infty,x,2n+1))$ and $\pi((x,\infty,2n))=\pi((x,\infty,2n-1))$ for any $n\in \mathbb Z$ and $x\in\Tau$ we obtain that $y_{2n-1}=y_{2n}=y_{2n+1}$, $n\in \mathbb Z$. Hence there exists a unique $y\in Y$ such that
$$D=\pi(\{(x,\infty,n)\mid n\in\mathbb{Z} \hbox{ and } x\in \Tau\setminus\{T_{\xi}\mid \xi\leq \alpha\}\})\subset f^{-1}(y).$$
Observe that $\{a,b\}\subset \overline{D}\subset \overline{f^{-1}(y)}=f^{-1}(y)$. Hence any $T_1$ space $Y$ of pseudocharacter $\leq\kappa$ belongs to $\operatorname{Const}_{a,b}(J(\Pi))$.

Finally, consider the space $D_{a,b}(J(\Pi))$ and put $Q=cW(D_{a,b}(J(\Pi)))$.
Similarly as in the proof of Theorem~\ref{t1} it can be checked that the space $Q$ is regular, separable, totally countably compact and $Y$-rigid.
\end{proof}

Next we show that the assumption of Theorem~\ref{t2} consistently holds for many
$\kappa<\mathfrak c$ simultaneously. This fact may be thought of as a folklore, e.g., in private communication we have been informed by V. Fischer and J. Schilchan that they
have come across the same ideas. Nonetheless we give a somewhat detailed proof thereof using the \emph{Mathias forcing $\mathbb M_{\mathcal F}$ associated to a free filter $\mathcal F$}
on $\omega$.
$\mathbb M_{\mathcal F}$ consists of pairs
$\langle s,F\rangle$ such that $s\in [\omega]^{<\omega}$, $F\in\mathcal F$, and $\max s<\min F$.
A condition $\langle s,F\rangle$ is stronger than $\langle s',F'\rangle$
if $F\subset F'$, $s$ is an end-extension of $s'$, and
$s\setminus s'\subset F'$.
This poset introduces the generic subset $X\in [\omega]^\omega$ such that $X\subset^* F$
for all $F\in\mathcal F$, namely $X=\bigcup\{s\: :\: \exists F\in\mathcal F\:( \langle s, F\rangle\in G) \}$,
where $G$ is $\mathbb M_{\mathcal F}$ generic.

\begin{proposition}\label{Zd}
 Let $V$ be a model of GCH and  $\kappa$  a cardinal of uncountable cofinality. Then there is a ccc poset $\mathbb P$
 such that in $V^{\mathbb P}$, $2^\omega=\kappa$ and for every regular cardinal $\lambda\leq \kappa$ there is a tower
 $\{T^\lambda_\xi:\xi<\lambda\}$.
\end{proposition}
\begin{proof}
$\mathbb P$ will be a two step iteration $\mathbb P_0*\dot{\mathbb P}_1$, where $\mathbb P_0=Fin(\kappa\times\omega,2)$ is the standard Cohen
poset adding $\kappa$-many Cohen reals. Let $\Lambda$ be the set of all regular cardinals $\lambda\leq\kappa$.
$\dot{\mathbb P}_1$ is a ($\mathbb P_0$-name for) the finitely supported product $\prod_{\lambda\in\Lambda}\mathbb Q^\lambda$, where $\mathbb Q^\lambda=\mathbb Q^\lambda_\lambda$ in  the finite support  iteration $\langle \mathbb Q^\lambda_\alpha,\dot{\mathbb R}^\lambda_\beta:\beta<\lambda, \alpha\leq\lambda\rangle$  defined as follows.
Suppose that we have already constructed $\langle \mathbb Q^\lambda_\alpha,\dot{\mathbb R}^\lambda_\beta:\beta<\xi, \alpha\leq\xi\rangle$ for some
$\xi<\lambda$, giving rise to a sequence
$\langle \dot{A}^\lambda_\beta:\beta<\xi\rangle$ of $\mathbb Q^\lambda_\xi$-names for infinite subsets of $\omega$ such that
$1_{\mathbb Q^\lambda_\xi}$ forces $\langle \dot{A}^\lambda_\beta:\beta<\xi\rangle$ to be    $\subset^*$-decreasing.
Then $\dot{\mathbb R}^\lambda_\xi$ is the $\mathbb Q^\lambda_\xi$-name for the Mathias forcing
$\mathbb M_{\dot{\mathcal F}^\lambda_\xi}$, where  $\dot{\mathcal F}^\lambda_\xi$ is the name for the filter generated
by $\langle \dot{A}^\lambda_\beta:\beta<\xi\rangle$, and $\dot{A}^\lambda_\xi$ is the $\mathbb M_{\dot{\mathcal F}^\lambda_\xi}$-generic real.
This completes our definition of $\mathbb P$. In $V^{\mathbb P_0}$, every $\mathbb Q^\lambda$ is $\sigma$-centered being
an iteration of $\sigma$-centered posets of length at most continuum, see \cite{Movf}. Thus $\mathbb P_1$ is ccc in  $V^{\mathbb P_0}$
being a finite support product of $\sigma$-centered posets, and hence $\mathbb P$ is ccc as well. A standard counting of nice names for reals
shows that $2^\omega=\kappa$ in $V^{\mathbb P}$.

Let $G$ be $\mathbb P$-generic filter and $A^\lambda_\beta:=(\dot{A}^\lambda_\beta)^G$ for all $\lambda\in\Lambda$ and $\beta<\lambda$.
Thus in $V[G]$, $\langle A^\lambda_\beta:\beta<\lambda\rangle$ has the property that $A^\lambda_\beta\subset^* A^\lambda_\xi$ for all  $\xi<\beta<\lambda$. Let $h^\lambda_\beta$  be the enumerating function of $A^\lambda_\beta$, i.e., the unique increasing bijection from
$\omega$ onto $A^\lambda_\beta$. Since for every $\lambda\in\Lambda$ the poset $\mathbb P$ can be written in the form $\mathbb P_0*(\prod_{\nu\in\Lambda\setminus\{\lambda\}}\dot{\mathbb Q}^\nu\times\dot{\mathbb Q}^\lambda)$, we conclude that the family
$\{h^\lambda_\beta:\beta<\lambda\}$ is unbounded in $V[G]$, and hence the family $\{A^\lambda_\beta:\beta<\lambda\}$ has no infinite pseudointersection.
It follows that $\{\omega\setminus A^\lambda_\beta:\beta<\lambda\}$ is a tower of length $\lambda$ in $V[G]$.
\end{proof}

\begin{remark}
It has been pointed out to us by the referee that Proposition~\ref{Zd} is a corollary of Theorem 3.3 from~\cite{NV}. Namely, this theorem allows us to produce a model of ZFC with $2^{\omega}=\kappa$ for any given in advance cardinal $\kappa$ of uncountable cofinality, such that for any regular $\lambda\leq \kappa$ there exists an unbounded and non-dominating set $\{f_{\alpha}^{\lambda}:\alpha<\lambda\}$, well-ordered by $\leq^*$. Given any witness $g_{\lambda}$ of the non-domination of the family $\{f_{\alpha}^{\lambda}:\alpha<\lambda\}$, it can be checked that $\{A_{\alpha}^{\lambda}:\alpha<\lambda\}\cap[\omega]^{\omega}$ is an unbounded tower, where $A_{\alpha}^{\lambda}=\{n\in\omega: f_{\alpha}^{\lambda}(n)> g_{\lambda}(n)\}$. We have nonetheless decided to include the proof of Proposition~\ref{Zd} in order to make the paper more self-contained.
\end{remark}

Theorem~\ref{t2} and Proposition~\ref{Zd} imply the following:
\begin{theorem}\label{tt2}
It is consistent with ZFC that for every cardinal $\kappa<\mathfrak{c}$ there exists a regular separable totally countably compact space $Q$ such that $Q$ is $Y$-rigid for any $T_1$ space $Y$ of pseudocharacter $\leq \kappa$.
\end{theorem}

 However we don't know an answer to the following problem:

\begin{problem}
(CH) Does there exists a regular separable totally countably compact space which is $Y$-rigid with respect to any $T_1$ space $Y$ of countable pseudocharacter?
\end{problem}

\begin{remark}
Theorem~\ref{tt2} cannot be generalized to the case $\kappa=\mathfrak{c}$. Indeed, let $Y$ be a regular separable space.
Then $\psi(Y)\leq 2^\omega=\mathfrak c$ by inequality 2.7, page 15 of~\cite{Ju}, and thus
the identity map $\operatorname{id}:Y\rightarrow Y$ is continuous and non-constant.
\end{remark}

\begin{remark} \label{rem_sep}
The answer to Problem~\ref{p2} in the case of separable spaces is negative. For this observe that the cardinality of any regular separable space is bounded by the cardinal $2^{\mathfrak{c}}$. Thus there exists a set $S$ consisting of separable regular spaces such that any separable regular space is homeomorphic to some element of $S$. Consider the Tychonoff product $\Pi S$. Obviously the space $\Pi S$ is regular and for each regular separable space $X$ there exists a non-constant (even injective) continuous map $f:X\rightarrow \Pi S$.
\end{remark}

\section{Around the problem of Nyikos}

First we answer in the affirmative Problem~\ref{p3}, the latter being variation of Nyikos' problem obtained by weakening the
disjunction of first countability and countable compactness to the sequential compactness.

\begin{theorem}\label{t4}
There exists a regular separable sequentially compact non-functionally Hausdorff space $S$ within ZFC.
\end{theorem}

\begin{proof}
Fix any maximal tower $\Tau=\{T_{\alpha}\mid \alpha\in \kappa\}$ on $\omega$ and let $Y$ be the space of Franklin and Rajagopalan which corresponds to the tower $\Tau$ (see the space $Y$ in the proof of Theorem~\ref{t2} or Example 7.1 from~\cite{int}). Recall that the space $Y$ is separable normal locally compact and sequentially compact. By $Y^*$ we denote the one point compactification of $Y$. Observe that the Tychonoff product of $Z=Y{\times}Y^*$ is separable regular and sequentially compact.

\begin{claim}
The space $Z$ is not normal.
\end{claim}
\begin{proof}
Recall that the space $Y$ contains a closed homeomorphic copy of the cardinal $\kappa$ and the space $Y^*$ contains  a closed homeomorphic copy of the ordinal $\kappa+1$. Therefore the space $Z$ contains a closed homeomorphic copy of the Tychonoff product $K=\kappa{\times}(\kappa+1)$. It remains to show that that the space $K$ is not normal. To derive a contradiction, assume that $K$ is normal. Consider the closed disjoint subsets $A=\{(\alpha,\alpha)\mid \alpha\in \kappa\}$ and $B=\{(\alpha,\kappa)\mid \alpha\in \kappa\}$ of $K$. Since $K$ is pseudocompact, Glicksberg's Theorem (see \cite[Exercise 3.12.20(c)]{Eng}) implies that $\beta(K)=\beta(\kappa){\times}\beta(\kappa+1)=(\kappa+1){\times}(\kappa+1)$. Since the space $K$ is normal, Corollary 3.6.4 from~\cite{Eng} provides that $\cl_{\beta(K)}(A)\cap \cl_{\beta(K)}(B)=\emptyset$. However, it is easy to see that $(\kappa,\kappa)\in \cl_{\beta(K)}(A)\cap \cl_{\beta(K)}(B)$ which implies a contradiction.
\end{proof}

Finally, using Jones machine (see the end of Section~\ref{V}) we construct the space $S=J(Z)$ which is regular separable sequentially compact but not functionally Hausdorff.
\end{proof}

Recall that $X$ is called a Nyikos space, if $X$ is regular separable first countable and countably compact.
The rest of this section is devoted to the  following theorem whose proof will be divided into several steps
and rely onto the example of a non-normal Nyikos space constructed by Nyikos and Vaughan in~\cite{NV} under $\mathfrak{t}=\omega_1$. We shall need  the set-theoretic assumption
 ``$\omega_1=\mathfrak{t}<\mathfrak{b}=\mathfrak{c}$ and there exists a simple $P_{\mathfrak{c}}$-point'' which we
 denote by  $\varheart$.

\begin{theorem}\label{t6}
$(\varheart)$ There exists an $\R$-rigid Nyikos space.
\end{theorem}

Theorem~\ref{t6} implies that the affirmative answer to the strongest version of Problem~\ref{p1} is consistent, as
well as gives a consistent counterexample to the Nyikos' problem which is $\R$-rigid.

 Recall that an ultrafilter $\mathcal F$ on $\omega$ is called a \emph{simple $P_\kappa$-point}, where $\kappa$ is a regular cardinal, if there exists a basis $\{F_\alpha:\alpha<\kappa\}$ of $\mathcal F$ such that $\{\omega\setminus F_{\alpha}: \alpha<\kappa\}$ is a tower of length $\kappa$.
The following fact shows that $\varheart$ is consistent with ZFC and seems to be a kind of folklore. We are grateful to A. Dow for sharing with us the proof sketched below.

\begin{proposition}\label{ZZ}
$\varheart$ is consistent.
\end{proposition}
\begin{proof}
Let $V$ be a model  of GCH and $\mathbb P_0$ be a poset adding $\omega_3$ Cohen subsets of $\omega_1$ with countable supports.
Thus $2^\omega=\omega_1$ and $2^{\omega_1}=\omega_3$ holds in $V^{\mathbb P_0}$. In $V^{\mathbb P_0}$, let $\mathbb P_1$ be a finite support iteration of
length $\omega_2$ of ccc posets of size $\omega_1$ producing a simple $P_{\omega_2}$-point and forcing in addition $\mathfrak b=\mathfrak c=\omega_2$.
For this it is enough that we take cofinally often Hechler forcing as well as the Mathias forcing for ultrafilters which are made always larger
by including into the next ones the Mathias generics for the previous ones.

Since $2^{\omega_1}=\omega_3>\mathfrak c=\omega_2$ in $V^{\mathbb P_0*\dot{\mathbb P}_1}, $ we conclude that
$\mathfrak p=\omega_1$ in this model because $2^\kappa=\mathfrak c$ for any infinite $\kappa<\mathfrak p$.
The latter statement is also well-known and  follows from the existence of a countable dense subset $Q$ of $2^\kappa$ as well as the fact that
each point $x$ of $2^\kappa$ is a limit of a sequence in $Q$ convergent to $x$ because $2^\kappa$ has character $\kappa<\mathfrak p$ (see \cite[p.130]{int}),
which gives that $2^\kappa\leq |Q|^\omega=\mathfrak c$.
\end{proof}

Observe that the extension $D_{a,b}(X)$ does not preserve first countability. Therefore we shall use a suitable modification of this extension defined below.
Let $X$, $a$, $b$, $\mathcal B(a)$, $\mathcal B(b)$, $Z$, $D$, $D_n$, $E$, $b^*$, $f:E\to\omega$, $\sim$ and $\pi$
be such as at the beginning of Section~\ref{V}. In addition, we shall assume that $X$ is first countable.  Next we introduce two kinds of sets in $Z$ which will be useful in the definition of the topology on $Z/_{\sim}$ which is weaker than that of $D_{a,b}(X)$.

$\bullet$ For any $n\in\omega$ and subsets $U\subset X_n\setminus\{a_n,b_n\},$ $V\subset X\setminus\{b\}$ such that $\operatorname{int}(V)\in\mathcal B(a)$ we shall inductively construct the subset $[U]_V\subset Z$ as follows: let $U_{(0)}=\emptyset$, $U_{(1)}=U$ and assume that for some $n\in\omega$ the sets $U_{(k)}$, $k\leq n$ are already constructed. Put $U^*_{(n)}= (U_{(n)}\setminus U_{(n-1)})\cap E$ and $U_{(n+1)}=\cup\{V_{k}\mid k\in f(U^*_{(n)})\}\cup U_{(n)}$. Finally, let $[U]_V=\cup_{n\in\omega}U_{(n)}$.

$\bullet$ For any subsets $U\subset X\setminus\{a\}$ and $V\subset X\setminus\{b\}$ such that $\operatorname{int}(U)\in\mathcal{B}(b)$ and $\operatorname{int}(V)\in\mathcal{B}(a)$ we shall inductively construct the subset $[U_{b}]_V\subset Z$ as follows: let $U_{(0)}=\emptyset$, $U_{(1)}=\cup_{n\in\omega}U_{n}$ and assume that for some $n\in\omega$ the sets $U_{(k)}$, $k\leq n$ are already constructed. Put $U^*_{(n)}= (U_{(n)}\setminus U_{(n-1)})\cap E$ and $U_{(n+1)}=\cup\{V_{k}\mid k\in f(U^*_{(n)})\}\cup U_{(n)}$. Finally, let $[U_{b}]_V=\cup_{n\in\omega}U_{(n)}$.

 Let $\E$ be the quotient set $Z/_{\sim}$ endowed with the topology $\tau$ generated by the base
\begin{equation*}
  \begin{aligned}
    \mathcal{B}(\tau)=\{\pi([U_{n}]_V)\mid & n\in\omega,\hbox{ } U\hbox{ is an open subset of }X\setminus\{a,b\}\hbox{ and }V\in\mathcal{B}(a)\}\cup\\
&\cup\{\pi([U_{b}]_V\mid U\in \mathcal{B}(b) \hbox{ and }V\in\mathcal{B}(a)\}.
  \end{aligned}
\end{equation*}
 Observe that if $U\subset X\setminus\{a,b\}$ is open and $V\in\mathcal B(a)$, then $[U_{n}]_V$ is an open saturated subset of $Z$ for each $n\in\omega$. Similarly, if $U\in \mathcal B(b)$ and $V\in\mathcal B(a)$, then $[U_{b}]_V$ is an open saturated subset of $Z$.
 Hence the topology $\tau$ is weaker than the quotient topology on $Z/_{\sim}$, that is the identity map $id: D_{a,b}(X)\rightarrow \E$ is continuous. This immediately yields that $\E$ is separable and any continuous map $g: \E\rightarrow Y\in \operatorname{Const}(X)_{a,b}$ is constant.

\begin{lemma}\label{fc}
The space $\E$ is first countable.
\end{lemma}

\begin{proof}
Let $\{A^{(n)}\mid n\in\omega\}\subset\mathcal B(a)$ and $\{B^{(n)}\mid n\in\omega\}\subset\mathcal B(b)$ be open neighborhood bases at the points $a\in X$ and $b\in X$, respectively.  It is straightforward to check that the family $\{\pi([B^{(n)}_b]_{A^{(n)}})\mid n\in\omega\}$ forms an open neighborhood base at the point $b^*\in \E$. Fix any $q\in \E\setminus\{b^*\}$ and notice that there exist $x\in X\setminus\{a,b\}$ and $n\in\omega$ such that $\pi^{-1}(q)\setminus\{a_n\mid n\in\omega\}=\{x_n\}$. Let $\{U^{(k)}\mid k\in \omega\}$ be an open neighborhood base at the point $x\in X$ such that $U^{(k)}\subset X\setminus\{a,b\}$ for any $k\in\omega$. It is easy to check that the family $\{\pi([U^{(k)}_n]_{A^{(k)}})\mid k\in\omega\}$ forms an open neighborhood base at the point $q\in \E$. Hence the space $\E$ is first countable.
\end{proof}

It remains to verify that $\E$ is regular. For this we need the following auxiliary lemma:

\begin{lemma}\label{cl}
For any $n\in\omega$ and subsets $W\in\mathcal{B}(b)$, $V\in\mathcal{B}(a)$, and $U\subset X\setminus\{a,b\}$ such that $a\not\in\cl_X(U)$ the following inclusions hold:
\begin{itemize}
\item [$(i)$] $\cl_{\E}(\pi([U_n]_V))\subset\pi([\cl_X(U)_n]_{\cl_X(V)})$;
\item [$(ii)$] $\cl_{\E}(\pi([W_b]_V))\subset\pi([\cl_X(W)_b]_{\cl_X(V)}).$
\end{itemize}
\end{lemma}
\begin{proof}
Fix any $n\in\omega$ and subsets $W\in\mathcal{B}(b)$, $V\in\mathcal{B}(a)$, $U\subset X\setminus\{a,b\}$ such that $a\not\in\cl_X(U)$. Let $q\in \E$ be such that $\pi^{-1}(q)\subset Z\setminus [\cl_X(U)_n]_{\cl_X(V)}$. Observe that the set $[\cl_X(U)_n]_{\cl_X(V)}$ is closed in $Z$. Two cases are possible:

 1. $q\neq b^*$.  Let $\{x_{m}\}=\pi^{-1}(q)\setminus\{a_{k}\mid k\in \omega\}$.
Fix an open subset $T_m\subset X_{m}\setminus ([\cl_X(U)_n]_{\cl_X(V)}\cup\{a_m,b_m\})$ which contains $x_{m}$,  and
an $O\in \mathcal{B}(a)$ such that $O\cap U=\emptyset$.  It can be checked that  $\pi([T_m]_{O})$ is an open neighborhood of $q$ in $\E$ which is disjoint from the set $\pi([U_n]_V)$, which yields $q\notin \cl_{\E}(\pi([U_n]_V))$.
\smallskip

2. $q=b^*$. Fix any $T\in\mathcal{B}(b)$ such that $T\subset X\setminus\cl_X(V)$ and $T\cap U=\emptyset$.
Let $O\in \mathcal{B}(a)$ be such that $O\cap U=\emptyset$.  A routine verifications show that $[T_b]_{O}\cap [U_n]_V=\emptyset$.
Therefore, the set $\pi([T_b]_{O})$ is an open neighborhood of $b^*$ in $\E$ which is disjoint with the set $\pi([U_n]_V)$, thus witnessing that $b^*\notin \cl_{\E}(\pi([U_n]_V))$. This proves the inclusion $(i)$.

The second inclusion can be proved analogously.
\end{proof}

\begin{lemma}\label{reg}
The space $\E$ is regular.
\end{lemma}

\begin{proof}
Fix any $q\in\E\setminus\{b^*\}$, and let $\pi^{-1}(q)\setminus\{a_k\mid k\in\omega\}=\{x_n\}$. Observe that
$$\{q\}=\cap\{\pi([U_n]_V) \mid \hbox{ } V\in\mathcal{B}(a),\hbox{ } x\in U\subset X\setminus\{a,b\} \hbox{ and }U \hbox{ is open in }X\}\qquad \hbox{and}$$
$$\{b^*\}=\cap\{\pi([U_b]_V)\mid V\in\mathcal{B}(a), U\in\mathcal{B}(b)\}.$$ Hence $\E$ is a $T_1$ space.

Let $\pi([U_n]_V)$ be any basic open neighborhood of the point $q$. Recall that $U$ is an open neighborhood of $x$ in $X$. By the regularity of $X$, there exists an open neighborhood $W\subset X$ of $x$ such that $\cl_X(W)\subset U$. Using one more time the regularity of $X$ find any $F\in\mathcal{B}(a)$ such that $\cl_X(F)\subset V$. Lemma~\ref{cl} implies that
$$q\in\pi([W_n]_F)\subset \cl_{\E}(\pi([W_n]_F))\subset\pi([\cl_X(W)_n]_{\cl_X(F)})\subset \pi([U_n]_V).$$

Fix any open neighborhood $\pi([U_b]_V)$ of $b^*$. By the regularity of $X$, there exist $W\in\mathcal{B}(b)$ and $F\in\mathcal{B}(a)$ such that $\cl_X(W)\subset U$ and $\cl_X(F)\subset V$. Lemma~\ref{cl} provides that
$$\cl_{\E}(\pi([W_b]_F))\subset \pi([\cl_X(W)_b]_{\cl_X(F)})\subset \pi ([U_b]_V).$$

Hence the space $\E$ is regular.
\end{proof}


Let $\heartsuit$ be the assertion ``$\mathfrak{b}=\mathfrak{c}$ and there exists a simple $P_{\mathfrak{c}}$-point''.
The following theorem will be crucial in our construction of $\R$-rigid Nyikos space.

\begin{theorem}\label{BZ}
$(\heartsuit)$ Let $X$ be a regular first countable space of cardinality $\kappa<\mathfrak{c}$. Then there exists a regular first countable countably compact space $O(X)$ which contains $X$ as a subspace.
\end{theorem}

\begin{proof}
Let $X$ be a regular first countable space such that $|X|<\mathfrak{c}$. If $X$ is countably compact, then put $O(X)=X$. Otherwise, let $\mathcal{D}=\{A\in [X]^{\omega}\mid A$ is closed and discrete$\}$ and fix any bijection $h: \mathcal{D}\cup [\mathfrak{c}]^{\omega}\rightarrow \mathfrak{c}$ such that $h(a)\geq \sup(a)$ for any $a\in [\mathfrak{c}]^{\omega}$. It is easy to see that such a bijection exists. Next, using ideas of Ostaszewski~\cite{Ost}, for every $\alpha\leq \mathfrak{c}$ we shall inductively construct a topology $\tau_{\alpha}$ on $X\cup\alpha$ by defining an open neighborhood base $\mathcal{B}^{\alpha}(x)$ at each point $x\in X\cup\alpha$. For convenience, by $Y_{\alpha}$ we denote the space $(X\cup \alpha,\tau_{\alpha})$. At the end, we will show that the space $O(X)=Y_{\mathfrak{c}}$ has the desired properties, in particular, $\mathcal{B}^{\mathfrak{c}}(x)$ is a countable base at $x$ for any $x\in Y_{\mathfrak{c}}$.

Let $Y_0=X$ and for each $x\in X$ fix an open neighborhood base $\mathcal{B}^0(x)=\{U^0_n(x)\mid n\in\omega\}$ at the point $x$ such that $U^0_0(x)=X$ and $\overline{U^0_{n+1}(x)}\subset U^0_{n}(x)$. Assume that for each $\alpha<\xi$ the spaces $Y_{\alpha}$ are already constructed by defining for each $x\in Y_{\alpha}$ the family $\mathcal{B}^{\alpha}(x)=\{U^{\alpha}_n(x)\mid n\in\omega\}$ which forms an open neighborhood base at $x$ in $Y_{\alpha}$. Moreover, additionally suppose that the spaces $Y_{\alpha}$ are $T_1$, $U^{\alpha}_{0}(x)=Y_{\alpha}$, $\cl_{Y_{\alpha}}(U^{\alpha}_{n+1}(x))\subset U^{\alpha}_{n}(x)$ and $U^{\alpha}_n(x)= U^{\beta}_n(x)\cap Y_\alpha$ for any $x\in Y_{\alpha}$ and $\alpha<\beta<\xi$. Then we have three cases to consider:
\begin{itemize}
\item[1)] $\xi=\eta+1$ for some $\eta\in\mathfrak{c}$ and $h^{-1}(\eta)$ already has an accumulation point in $Y_{\eta}$;
\item[2)] $\xi=\eta+1$ for some $\eta\in\mathfrak{c}$ and $h^{-1}(\eta)$ does not have an accumulation point in $Y_{\eta}$;
\item[3)] $\xi$ is a limit ordinal.
\end{itemize}

1) Let $\mathcal{B}^{\xi}(x)=\mathcal{B}^{\eta}(x)$ for each $x\in Y_{\eta}$. The open neighborhood base $\mathcal{B}^{\xi}(\eta)$ at the point $\eta\in Y_{\xi}$ is defined as follows: $U_0^{\xi}(\eta)=Y_{\xi}$ and $U_n^{\xi}(\eta)=\{\eta\}$ for each $n>0$.

2) Let $h^{-1}(\eta)=\{z_n\}_{n\in\omega}$. Fix any simple $P_{\mathfrak{c}}$-point $p$ which exists by the assumption. Since the set $h^{-1}(\eta)$ is closed and discrete, for any $y\in y_{\eta}$ there exists a positive integer $k(y)$ such that  $|U_{k(y)}^{\eta}(y)\cap h^{-1}(\eta)|\leq 1$. Since $U_0^{\eta}(y)=X$ and the ultrafilter $p$ is free, for any $y\in Y_{\eta}$ there exists $m(y)<k(y)$ such that $F_y=\{n\in\omega\mid z_n\in U_{m(y)}^{\eta}(y)\setminus U^{\eta}_{m(y)+1}(y)\}\in p$.  Since $|Y_{\eta}|<\mathfrak{c}$ there exists $F_{\eta}\in p$ such that $F_{\eta}\subset^{*}F_{y}$ for any $y\in Y_{\eta}$. Let $d_{\eta}=\{z_n\mid n\in F_{\eta}\}$. Note that $$d_{\eta}\subset^{*}U^{\eta}_{m(y)}(y)\setminus U^{\eta}_{m(y)+1}(y)\subset U^{\eta}_{m(y)}(y)\setminus \cl_{Y_{\eta}}(U^{\eta}_{m(y)+2}(y))$$for any $y\in Y_{\eta}$. Since the set $U^{\eta}_{m(y)}(y)\setminus\cl_{Y_{\eta}}(U^{\eta}_{m(y)+2}(y))$ is open there exists a function $f_y\in\omega^{\omega}$ such that $\cl_{Y_{\eta}}(U^{\eta}_{f_{y}(n)}(z_n))\subset U^{\eta}_{m(y)}(y)\setminus\cl_{Y_{\eta}}(U^{\eta}_{m(y)+2}(y))$ for all but finitely many $n\in F_{\eta}$. Let $f_{\eta}\in\omega^{\omega}$ be a function such that $f_{\eta}\geq^*f_y$ for each $y\in Y_{\eta}$. Such an $f_{\eta}$ exists because $|Y_{\eta}|<\mathfrak{b}=\mathfrak{c}$. Next we define the open neighborhood base at the point $\eta\in Y_{\xi}$: Put $U_0^{\xi}(\eta)=Y_{\xi}$ and  $U^{\xi}_k(\eta)=\cup\{U^{\eta}_{f_{\eta}(n)+k}(z_n)\mid n\in F_{\eta}\setminus k\}\cup\{\eta\}$ for all $k\geq 1$. For each $y\in Y_{\eta}$ and $n\in\omega$ let $U_{n}^{\xi}(y)=U_n^{\eta}(y)$ if $n\geq m(y)+1$ and $U_{n}^{\xi}(y)=U_n^{\eta}(y)\cup\{\eta\}$ if $n\leq m(y)$. Obviously, the space $Y_{\xi}$ is $T_1$.
We claim that the family $\{U^{\eta}_{f_{\eta}(n)}(z_n)\mid n\in F_{\eta}\}$ is locally finite in $Y_{\eta}$.
Recall that for each $y\in Y_{\eta}$, $\cl_{Y_{\eta}}(U^{\eta}_{f_{y}(n)}(z_n))\subset U^{\eta}_{m(y)}(y)\setminus\cl_{Y_{\eta}}(U^{\eta}_{m(y)+2}(y))$ for all but finitely many $n\in F_{\eta}$. Since $f_{\eta}\geq^* f_y$ for any $y\in Y_{\eta}$ we obtain that $\cl_{Y_{\eta}}(U^{\eta}_{f_{\eta}(n)}(z_n))\subset U^{\eta}_{m(y)}(y)\setminus\cl_{Y_{\eta}}(U^{\eta}_{m(y)+2}(y))$ for all but finitely many $n\in F_{\eta}$. Hence for any $y\in Y_{\eta}$ the set $U^{\eta}_{m(y)+2}(y)$ intersects only finitely many elements of the family $\{U^{\eta}_{f_{\eta}(n)}(z_n)\mid n\in F_{\eta}\}$ witnessing that this family is locally finite. Fix any $k\in\omega$. Since $U^{\eta}_{f_{\eta}(n)+k}(z_n)\subset U^{\eta}_{f_{\eta}(n)}(z_n)$ we obtain that the family $\{U^{\eta}_{f_{\eta}(n)+k}(z_n)\mid n\in F_{\eta}\}$ is locally finite as well. The following proves that $Y_\xi$ is Hausdorff: for every $y\in Y_\eta$ the intersection
$U^\xi_{m(y)+2}(y)\cap U^\xi_k(\eta)=U^\eta_{m(y)+2}(y)\cap U^\xi_k(\eta)$ is empty provided that $k$ is larger than all of the (finitely many)
elements  $n\in F_\eta$ such that $U^\eta_{f_y(n)}(z_n)\cap U^\xi_{m(y)+2}(y)\neq\emptyset$.

Next, we shall  prove that $\cl_{Y_{\xi}}(U^{\xi}_{n+1}(y))\subset U^{\xi}_{n}(y)$ for each $y\in Y_{\xi}$. First consider the  case
$y\in Y_\eta$. If $n\geq m(y)+1$, then
$$\cl_{Y_{\xi}}(U^{\xi}_{n+1}(y))=\cl_{Y_{\xi}}(U^{\eta}_{n+1}(y))=\cl_{Y_{\eta}}(U^{\eta}_{n+1}(y))\subset U^\eta_n(y)\subset U^\xi_n(y),$$
because in this case $n+1\geq m(y)+2$ and thus $\eta$, the only element of $Y_\xi\setminus Y_\eta$, has a neighbourhood disjoint from
$U^\eta_{n+1}(y)$, as we have established above. If $n\leq m(y)$, then we have
\begin{eqnarray*}
\cl_{Y_{\xi}}(U^{\xi}_{n+1}(y))\subset \cl_{Y_{\xi}}(U^{\eta}_{n+1}(y)\cup\{\eta\})=\cl_{Y_{\xi}}(U^{\eta}_{n+1}(y))\cup\{\eta\}= \cl_{Y_{\eta}}(U^{\eta}_{n+1}(y))\cup\{\eta\}\subset \\
\subset U^\eta_n(y)\cup\{\eta\} = U^\xi_n(y).
\end{eqnarray*}~The second equality is again a consequence of the fact that $\{\eta\}=Y_\xi\setminus Y_\eta$.

Now let us consider the case $y=\eta$.
For each $k\in\omega$ we have
\begin{eqnarray*}
 \cl_{Y_\xi}(U^{\xi}_{k+1}(\eta))= \cl_{Y_\xi}\big(\bigcup\{U^{\eta}_{f_{\eta}(n)+k+1}(z_n) \mid n\in F_{\eta}\setminus (k+1)\}\cup\{\eta\}\big)=\\
=  \cl_{Y_\xi}\big(\bigcup\{U^{\eta}_{f_{\eta}(n)+k+1}(z_n) \mid n\in F_{\eta}\setminus (k+1)\}\big)\cup\{\eta\}=\\
= \cl_{Y_\eta}\big(\bigcup\{U^{\eta}_{f_{\eta}(n)+k+1}(z_n) \mid n\in F_{\eta}\setminus (k+1)\}\big)\cup\{\eta\}=\\
=\bigcup\{\cl_{Y_\eta}\big(U^{\eta}_{f_{\eta}(n)+k+1}(z_n)\big) \mid n\in F_{\eta}\setminus (k+1)\}\cup\{\eta\}\subset\\
\subset \bigcup\{U^{\eta}_{f_{\eta}(n)+k}(z_n) \mid n\in F_{\eta}\setminus (k+1)\}\cup\{\eta\}\subset U^\xi_{k}(\eta).
\end{eqnarray*}
The latest equality is a consequence of the local finiteness of the family
$\{U^{\eta}_{f_{\eta}(n)+k+1}(z_n) \mid n\in F_{\eta}\}$ in $Y_\eta$.

3)  For each $y\in Y_\xi$ put $U^{\xi}_n(y)=\cup_{\alpha(y)<\gamma<\xi}U^{\gamma}_n(y)$, where $\alpha(y)=\min\{\alpha<\xi:y\in Y_\alpha\}$.
The next claim implies that the space $Y_{\xi}$ is Hausdorff.
\begin{claim}\label{a}
For any $\gamma<\xi$, $n,m\in\omega$ and distinct points $y_0,y_1\in Y_{\gamma}$, if $U_n^{\gamma}(y_0)\cap U_m^{\gamma}(y_1)=\emptyset$, then $U_n^{\xi}(y_0)\cap U_m^{\xi}(y_1)=\emptyset$.
\end{claim}
\begin{proof}
To derive a contradiction, assume that  $U_n^{\xi}(y_0)\cap U_m^{\xi}(y_1)\neq\emptyset$. It is easy to see that $U_n^{\xi}(y_0)\cap U_m^{\xi}(y_1)\subset [\gamma,\xi)$. Let $\delta=\min U_n^{\xi}(y_\xi)\cap U_m^{\xi}(y_1)$. It follows that $U_n^{\delta}(y_0)\cap U_m^{\delta}(y_1)=\emptyset$ and $U_n^{\delta+1}(y_0)\cap U_m^{\delta+1}(y_1)=\{\delta\}$. Then the set $d_{\delta}$ (see case 2 above) is closed and discrete in $Y_{\delta}$.  Since $\delta\in U_{n}^{\delta+1}(y_0)\cap U_{m}^{\delta+1}(y_1)$, the definition of $U_{n}^{\delta+1}(y_0)$ and $U_{m}^{\delta+1}(y_1)$ implies that $d_{\delta}\subset^* U_{n}^{\delta}(y_0)\cap U_{m}^{\delta}(y_1)\neq \emptyset$. But this contradicts our assumption.
\end{proof}
It remains to check that $\cl_{Y_{\xi}}(U^{\xi}_{n+1}(y))\subset U^{\xi}_{n}(y)$ for each $y\in Y_{\xi}$.
Fix any $y\in Y_{\xi}$ and $n\in\omega$. By the assumption, $\cl_{Y_{\gamma}}U^{\gamma}_{n+1}(y)\subset U^{\gamma}_{n}(y)$ for each $\gamma<\xi$. To derive a contradiction, assume that there exists $z\in Y_{\xi}$ such that $z\in \cl_{Y_{\xi}}U^{\xi}_{n+1}(y)\setminus U^{\xi}_{n}(y)$. Fix any $\delta<\xi$ such that $z\in Y_{\delta}$. We claim that $z\in \cl_{Y_{\delta}}U^{\delta}_{n+1}(y)$. Indeed, if $z\notin \cl_{Y_{\delta}}(U^{\delta}_{n+1}(y))$, then there exists $m\in\omega$ such that $U_{m}^{\delta}(z)\cap U^{\delta}_{n+1}(y)=\emptyset$. The above Claim implies that $U_m^{\xi}(z)\cap U_{n+1}^{\xi}(y)=\emptyset$ witnessing that $z\notin \cl_{Y_{\xi}}(U^{\xi}_{n+1}(y))$. The obtained contradiction implies that $z\in \cl_{Y_{\delta}}(U^{\delta}_{n+1}(y))\subset U^{\delta}_{n}(y)\subset U^{\xi}_n(y)$, which contradicts the choice of $z$.
Hence $\cl_{Y_{\xi}}(U^{\xi}_{n+1}(y))\subset U^{\xi}_{n}(y)$ for each $y\in Y_{\xi}$.

Let $O(X)=Y_{\mathfrak{c}}$. By the construction, the space $O(X)$ is regular and first countable. Let $A$ be any countable subset of $O(X)$. If the set $B=A\cap \mathfrak{c}$ is infinite, then consider $h(B)\in\mathfrak{c}$. By the construction, either $B$ has an accumulation point in $Y_{h(B)}$ or $h(B)$ is an accumulation point of $B$ in $Y_{h(B)+1}$. In both cases $B$ has an accumulation point in $O(X)$. If $A\subset^* X$, then either it has an accumulation point in $X$, or $A$ is closed and discrete in $X$. In the latter case either $A$ has an accumulation point in $Y_{h(A)}$ or $h(A)$ is an accumulation point of $A$ in $Y_{h(A)+1}$. Thus $O(X)$ is countably compact.
\end{proof}

We are in a position now to present the
\medskip

\noindent\textit{Proof of Theorem~\ref{t6}.} \
Let $X$ be regular non-normal Nyikos space constructed by Nyikos and Vaughan under $\mathfrak{t}=\omega_1$ in~\cite{NV}. Note that $|X|=\omega_1$. Applying Jones machine to $X$ we obtain a Nyikos space $J(X)$ of cardinality $\omega_1$ containing two points $a,b$  which cannot be separated by any real-valued continuous map. Then consider the space $E_{a,b}(J(X))$ which is regular separable first countable $\R$-rigid and $|E_{a,b}(J(X))|=\omega_1$. Theorem~\ref{BZ} implies that the space $O(E_{a,b}(J(X)))$ is regular first countable and countably compact. Then $\overline{E_{a,b}(J(X))}\subset O(E_{a,b}(J(X)))$ is an $\R$-rigid Nyikos space.
\hfill $\Box$

\begin{remark}
Example 2 from~\cite{BBR} shows that there exists a regular first countable space of cardinality $\mathfrak{d}$ which cannot be embedded into Urysohn countably compact spaces. Also, it can be proved that for any mad family $A$ on $\omega$ the Mr\'owka-Isbell space $\Psi(A)$ is Tychonoff first countable and cannot be embedded into any Hausdorff countably compact space of character $<\mathfrak{b}$.
Hence Theorem~\ref{BZ} cannot be proved within ZFC. In particular its conclusion cannot hold in any model of $\min\{\mathfrak{d},\mathfrak{a}\}<\mathfrak{c}$.
\end{remark}

Theorems~\ref{t4} and \ref{t6} motivate the following

\begin{problem}
Does there exist a ZFC example of a separable sequentially compact regular $\R$-rigid space?
\end{problem}

By $d(X)$, $w(X)$, $\chi(X)$, $\psi(X)$, $s(X)$, $e(X)$, and $c(X)$ is denoted the
density,  weight, character, pseudocharacter, spread, extent, and cellularity
of the space $X$, respectively. We refer the reader to~\cite{Eng} for their definitions.
We end up with a remark which shows how the extensions $D_{a,b}(X)$ and $E_{a,b}(X)$ can be generalized for a non-separable space $X$ and reveals some of their properties. Its proof is fairly standard and is therefore left to the reader.

\begin{remark}
Let $a,b\in X$, $\mathcal B(a)$, and $\mathcal B(b)$ be such as in Section~\ref{V}.
  Let $Z$ be the Tychonoff product $X{\times} d(X)$ where $d(X)$ is endowed with the discrete topology. For any $x\in X$ and $\alpha\in d(X)$ by $x_{\alpha}$ we denote the point $(x,\alpha)$. Analogously, for any $B\subset X$ and $\alpha\in d(X)$ the set $B{\times}\{\alpha\}$ is denoted by $B_{\alpha}$. Let $D$ be a dense subset of $X$ such that $D\subset X\setminus\{a,b\}$ and $|D|=d(X)$. Put $E=\cup_{\alpha\in d(X)}D_{\alpha}$. Fix any bijection $f:E\rightarrow d(X)$ such that $f(x_{\alpha})\neq \alpha$ for each $\alpha\in d(X)$ and $x_{\alpha}\in D_{\alpha}$. On the set $Z$ consider the smallest equivalence relation $\sim$ which satisfies the following conditions:
\begin{itemize}
\item $x_{\alpha}\sim a_{f(x_{\alpha})}$ for any $\alpha\in d(X)$ and $x\in D$;
\item $b_{\alpha}\sim b_{\xi}$ for any $\alpha,\xi\in d(X)$.
\end{itemize}
Let $\D$ be the quotient space $Z/_{\sim}$.  Analogously, one can generalize the extension $\E$ for non-separable spaces.

Let $Z\in\{\D,\E\}$.
Then every continuous map $g:Z\rightarrow Y\in \operatorname{Const}(X)_{a,b}$ is constant.
The space $\D$ is regular and $d(\D)=d(X)$. If the space $X$ is countably compact, then $\D$ has property $D$. Also, analogues of Lemma~\ref{l2} and Corollary~\ref{c1} hold for the generalized extension $\D$.

The space $\E$ is also regular and $d(\E)=d(X)$,
$w(\E)=w(X)$,
$\chi(\E)=\chi(X)$,
$\psi(\E)=\psi(X)$,
$s(\E)= \max\{s(X),d(X)\}$,
 $e(\E)=\max\{e(X),d(X)\}$, and
$c(\E) = \max\{c(X),d(X)\}$.
\end{remark}

\section*{Acknowledgements}
The authors acknowledge the Referee for numerous valuable remarks, in particular, for the idea of generalizing Lemma~\ref{CC} which significantly improves Theorem~\ref{t2}, and Piotr Borodulin-Nadzieja for his comments on Question~\ref{p3}. In parti\-cu\-lar, Piotr convinced us that this question cannot be solved using tight gaps.

\end{document}